\documentclass[11pt,a4paper]{article}
\usepackage[margin=25mm]{geometry}
\usepackage{amsmath,amssymb,amsthm}
\usepackage{microtype,xurl,needspace}
\usepackage[hidelinks]{hyperref}
\newtheorem{theorem}{Theorem}[section]
\newtheorem{proposition}[theorem]{Proposition}
\newtheorem{lemma}[theorem]{Lemma}
\newtheorem{corollary}[theorem]{Corollary}
\theoremstyle{definition}
\newtheorem{definition}[theorem]{Definition}
\newtheorem{example}[theorem]{Example}
\newtheorem{problem}[theorem]{Problem}
\theoremstyle{remark}

\newcommand{\R}{\mathbb R}
\newcommand{\Z}{\mathbb Z}
\newcommand{\id}{\operatorname{id}}
\newcommand{\diam}{\operatorname{diam}}
\title{Nonexpansive Bijections in $T_0$-Quasi-Metric Spaces and Their $q$-Hyperconvex Hulls}
\author{Collins Amburo Agyingi\\
Department of Mathematical Sciences, University of South Africa\\
P.O. Box 392, UNISA 003, Pretoria, South Africa\\
\texttt{agyinca@unisa.ac.za}}
\date{}
\begin{document}
\maketitle
\begin{abstract}
We study bijective nonexpansive maps in $T_0$-quasi-metric spaces, distinguishing preservation of directed distances from pointwise rigidity. Our focus is their behaviour under passage to the $q$-hyperconvex hull. The upper quasi-metric integers are $q$-plastic and bicomplete, but their hull is the non-$q$-plastic upper real line. An explicit hull calculation and an established metric dense-rigidity theorem give the same failure of inheritance for a $q$-rigid original space. Tightness propagates distance preservation from the canonical copy throughout the hull. Under compactness and a bounded lattice structure for the zero-distance order, identifying that copy with the complemented elements gives a sufficient condition for inheritance of $q$-rigidity. A four-vertex asymmetric rectangle verifies this criterion. We also record metric transfer mechanisms and show that a directed interval condition gives an endpoint-coordinate representation, automatic $q$-plasticity, and $q$-rigidity under uniqueness of the ordered diametral pair.
\end{abstract}

\section{Introduction}
Plasticity asks whether a bijective nonexpansive self-map must preserve
distances. Every totally bounded metric space is plastic \cite{Naimpally}, while noncompact spaces admit more delicate behaviour. Banakh, Mazurenko and Zavarzina \cite{Banakh} establish nonplasticity of countable dense subsets of normed spaces and existence of dense rigid subspaces in separable metric spaces satisfying a compact-richness condition. Their results separate the dynamics of an ambient space from those of its dense subspaces.

Quasi-metric plasticity has already been studied by Zavarzina \cite{Zavarzina}, including the role of precompactness and hereditary precompactness. Her separation convention requires $d(x,y)=0$ to imply $x=y$. Here we use the weaker $T_0$ condition $d(x,y)=d(y,x)=0\Rightarrow x=y$, allowing distinct points to have zero distance in one direction. We call a space $q$-plastic when every bijective nonexpansive self-map preserves directed distances, and $q$-rigid when every such map is the identity.

Our main question concerns inheritance by the $q$-hyperconvex hull. The
Isbell-hull construction for di-spaces was developed by Kemajou, K\"unzi
and Olela Otafudu \cite{Kemajou}; related hyperconvexity and tightness
results appear in \cite{Hyperconvex,Tight}. Tightness forces a nonexpansive
hull map to preserve all distances once it does so on the canonical copy.
It does not force a bijective map to preserve that copy set-wise.

We exhibit this obstruction on the upper quasi-metric line. Every subset
unbounded in both directions has the whole upper real line as its hull.
The integers are $q$-plastic and bicomplete, whereas the hull admits bijective
contractions. Combining the hull calculation with metric dense rigidity
gives a $q$-rigid space with a non-$q$-plastic hull. In the positive direction,
a compact hull whose zero-distance order is a bounded lattice inherits
$q$-rigidity when the canonical copy consists exactly of the complemented
elements. A four-vertex rectangle verifies the structural hypotheses.
Its rectangular hull geometry is already known
\cite[Example~2]{EndpointsII}; the application here concerns recognition
of the canonical copy and rigidity.

The focus is therefore failure of inheritance and recognition of the
canonical copy. The symmetrization, dense-rigidity and tightness arguments
are established mechanisms supporting these examples. In particular, the sum-metric argument of \cite{Zavarzina} gives $q$-plasticity when the
maximum symmetrization is totally bounded. The contraction construction
below is a conditional directional adaptation of \cite{Banakh}, rather than a theorem that all asymmetric normed spaces satisfy the required homogeneity property.

Directed intervals provide a further illustration. Two forward endpoint
distances recover the quasi-metric and embed the space into a bounded upper
quasi-metric plane. Total boundedness gives $q$-plasticity automatically;
endpoint recognition then gives pointwise rigidity. Section~\ref{sec:background}
collects the background and metric transfers, Section~\ref{sec:hulls}
contains the hull results, and Section~\ref{sec:intervals} treats directed
intervals. Section~\ref{sec:questions} records the remaining questions.

\section{Preliminaries and metric transfer mechanisms}
\label{sec:background}
All quasi-metrics are finite-valued. A quasi-metric on a nonempty set $X$
is a function $d:X\times X\to[0,\infty)$ satisfying $d(x,x)=0$ and
$d(x,z)\le d(x,y)+d(y,z)$. It is $T_0$ if $d(x,y)=d(y,x)=0$ implies $x=y$. Its conjugate and maximum symmetrization
are
\[
d^{-1}(x,y)=d(y,x),\qquad d^s(x,y)=\max\{d(x,y),d(y,x)\}.
\]
The latter is a metric in the $T_0$ case. Unless stated otherwise, density,
separability, compactness and completeness refer to $d^s$. A space is
\emph{bicomplete} when $d^s$ is complete. Background terminology follows
\cite{Cobzas,Kunzi}.

A map $F:(X,d_X)\to(Y,d_Y)$ is \emph{nonexpansive} if $d_Y(Fx,Fy)\le d_X(x,y)$ for every ordered pair; it is an \emph{isometric embedding} if equality always holds. A bijective isometric embedding is an isometry. Nonexpansiveness for directed distances implies nonexpansiveness for the maximum symmetrizations by applying the inequalities to both orientations.

\begin{definition}
A $T_0$-quasi-metric space is \emph{$q$-plastic} if every bijective
nonexpansive self-map is an isometry. It is \emph{$q$-rigid} if every such map
is the identity. Thus $q$-rigidity implies $q$-plasticity; for symmetric $d$,
$q$-plasticity is metric plasticity.
\end{definition}

\subsection{Symmetrization and bicompletion}
The triangle inequality gives the useful continuity estimate
\begin{equation}\label{eq:continuity}
|d(x,y)-d(x',y')|\le d^s(x,x')+d^s(y,y').
\end{equation}
Indeed, apply the triangle inequality through $x',y'$ and then interchange
the two pairs. On the metric completion $\widehat X$ of $(X,d^s)$,
this estimate extends $d$ uniquely to a continuous function $\widehat d$.
Approximation by sequences gives its triangle inequality and
$\max\{\widehat d(u,v),\widehat d(v,u)\}=\widehat{d^s}(u,v)$.
Hence $\widehat d$ is a bicomplete $T_0$-quasi-metric. Likewise, a
nonexpansive map from a dense subspace into a bicomplete space extends
uniquely to the bicompletion: extend for the symmetrized metrics and use
\eqref{eq:continuity} to pass the directed inequalities to the limit.
These are standard completion mechanisms; see \cite{Alemany,Cobzas}.

\begin{proposition}\label{prop:totalbounded}
If $(X,d^s)$ is totally bounded, then $(X,d)$ is $q$-plastic.
\end{proposition}
\begin{proof}
Set $\rho(x,y)=d(x,y)+d(y,x)$. This is a metric with
$d^s\le\rho\le2d^s$, so it is totally bounded. A bijective
$d$-nonexpansive self-map $F$ is $\rho$-nonexpansive. Metric plasticity
\cite{Naimpally} yields
\[
d(Fx,Fy)+d(Fy,Fx)=d(x,y)+d(y,x).
\]
Each summand on the left is at most the corresponding summand on the right;
equality forces equality of both. This is the sum-metric mechanism used in
\cite{Zavarzina}, here formulated under $T_0$ separation.
\end{proof}

\begin{proposition}[Metric dense-rigidity transfer]\label{prop:dense}
Suppose $(Y,d^s)$ is separable and every nonempty $d^s$-open subset
contains an uncountable $d^s$-compact subset. Then $Y$ contains a
$d^s$-dense $q$-rigid subspace.
\end{proposition}
\begin{proof}
The metric theorem of Banakh, Mazurenko and Zavarzina \cite{Banakh}
provides a dense subspace $A$ on which every bijective $d^s$-nonexpansive
self-map is the identity. Every $d$-nonexpansive map is also
$d^s$-nonexpansive, so $A$ is $q$-rigid.
\end{proof}
For brevity, call the compact-richness hypothesis \emph{c-crowdedness}.
This proposition is a transfer result, not a new construction. It also
applies to $(Q(X),D)$ whenever its symmetrization is separable and c-crowded.
The separability formulation can equivalently use the characterization of
supseparability in \cite{Supseparability}.

\subsection{A conditional dense-subspace contraction argument}
Following the metric Lipschitz homogeneity framework
\cite{Dijkstra,Banakh}, call $(Y,d)$ \emph{q-LCDH} if, for every
$0<\varepsilon<1$ and every pair of countable dense subsets $A,B$,
there is a bijection $g:Y\to Y$ with $g[A]=B$ and
\[
(1-\varepsilon)d(x,y)\le d(gx,gy)\le(1+\varepsilon)d(x,y)
\qquad(x,y\in Y).
\]
These inequalities control both orientations and make $g$ bi-Lipschitz
for $d^s$. A contraction is a map $T$ for which
$d(Tx,Ty)\le c\,d(x,y)$ with some $0<c<1$.

\begin{proposition}\label{prop:contraction}
If a q-LCDH space $Y$ admits a bijective contraction, then every countable
dense subset $A\subseteq Y$ with at least two points admits a bijective
contraction and is not $q$-plastic. In particular, this holds for a countable
space whose bicompletion has those two properties.
\end{proposition}
\begin{proof}
Let $T$ have contraction constant $c$. It is continuous for $d^s$, and
surjectivity makes $T[A]$ dense: for $y=Tz$, approximate $z$ by points
of $A$. Choose $\varepsilon$ with $(1+\varepsilon)c<1$ and use q-LCDH
to obtain $g[T[A]]=A$. Then $g\circ T|_A$ is bijective and
\[
d(gTx,gTy)\le(1+\varepsilon)c\,d(x,y).
\]
The $T_0$ condition supplies a positive directed distance between two
distinct points, which this map strictly decreases. For the final assertion,
apply the argument to the canonical dense copy in the bicompletion.
\end{proof}

No continuity of $T^{-1}$ is needed. The result remains conditional:
q-LCDH must be verified in an application. If $(E,p)$ is an asymmetric
normed real vector space with q-LCDH bicompletion, it follows that each
countable $p^s$-dense subset with at least two points is not $q$-plastic.
Indeed, $p^s(v)=\max\{p(v),p(-v)\}$ is a norm and
$|p(v)-p(w)|\le p^s(v-w)$ extends $p$ to its norm completion.
Subadditivity and positive homogeneity pass to the extension $\widehat p$,
and $\max\{\widehat p(v),\widehat p(-v)\}=\widehat{p^s}(v)$.
This gives the bicompletion with distance $\widehat p(w-v)$;
$v\mapsto v/2$ is a bijective contraction there.

\section{Plasticity and rigidity of $q$-hyperconvex hulls}\label{sec:hulls}
\subsection{The minimal-pair model}
A space $(M,d)$ is \emph{$q$-hyperconvex} if each family
$(x_i,r_i,s_i)_{i\in I}$ with nonnegative radii and
$d(x_i,x_j)\le r_i+s_j$ satisfies
\[
\bigcap_{i\in I}\bigl(B_d[x_i,r_i]\cap B_{d^{-1}}[x_i,s_i]\bigr)
\ne\varnothing,
\quad B_d[x,r]=\{y:d(x,y)\le r\}.
\]
The hull theory and its function-pair model are due to
\cite{Kemajou}; related results appear in \cite{Hyperconvex,Tight,Willerton},
and stability is studied in \cite{Zava}.

A pair $f=(f_1,f_2)$ of nonnegative functions on $X$ is \emph{ample} if
$d(x,y)\le f_2(x)+f_1(y)$ for all $x,y$. Order ample pairs
coordinatewise and let $Q(X)$ consist of the minimal ones. With
$r^+=\max\{r,0\}$, its quasi-metric is
\begin{equation}\label{eq:hullmetric}
D(f,g)=\max\left\{\sup_{x\in X}(f_1(x)-g_1(x))^+,
                       \sup_{x\in X}(g_2(x)-f_2(x))^+\right\}.
\end{equation}
Minimality is equivalent to the extremality identities
\begin{equation}\label{eq:extremal}
f_1(x)=\sup_{y\in X}(d(y,x)-f_2(y))^+,\qquad
f_2(x)=\sup_{y\in X}(d(x,y)-f_1(y))^+.
\end{equation}
They imply $(f_1(y)-f_1(x))^+\le d(x,y)$ and
$(f_2(y)-f_2(x))^+\le d(y,x)$. In particular, for a fixed $x$,
\[
f_1(y)-g_1(y)\le f_1(x)+d(x,y)-g_1(y)\le f_1(x)+g_2(x),
\]
and similarly $g_2(y)-f_2(y)\le g_2(x)+f_1(x)$.
Thus $D(f,g)\le f_1(x)+g_2(x)<\infty$, and
\[
D^s(f,g)=\max\{\|f_1-g_1\|_\infty,\|f_2-g_2\|_\infty\}.
\]
The canonical embedding and evaluation identities are
\begin{equation}\label{eq:evaluation}
e(x)=(d(x,\cdot),d(\cdot,x)),\qquad
D(f,e(x))=f_1(x),\quad D(e(x),f)=f_2(x).
\end{equation}
For example, the second supremum in $D(f,e(x))$ is $f_1(x)$ by
\eqref{eq:extremal}, and the first is at most $f_1(x)$ by the
directional Lipschitz inequality. The other identity is analogous.

The space $(Q(X),D)$ is $q$-hyperconvex and is a tight extension of $e(X)$
\cite{Kemajou,Tight}. Here tightness means that a quasi-pseudometric
$\delta\le D$ agreeing with $D$ on $e(X)\times e(X)$ must equal $D$.
A $q$-hyperconvex tight extension is a hull, uniquely up to an isometry
identifying the original copies. These established facts are the mechanisms
used below. Proposition~\ref{prop:totalbounded} also shows that a compact
symmetrization makes the hull $q$-plastic.

\subsection{Inheritance when the canonical copy is invariant}
\begin{lemma}\label{lem:relative}
A nonexpansive map $H:Q(X)\to Q(X)$ fixing $e(X)$ pointwise is the
identity.
\end{lemma}
\begin{proof}
Write $H(f)=g$. Nonexpansiveness and \eqref{eq:evaluation} give
$g_1(x)\le f_1(x)$ and $g_2(x)\le f_2(x)$ for every $x$.
The pair $g$ is ample, so minimality of $f$ gives $g=f$.
\end{proof}

\begin{lemma}\label{lem:propagation}
If a nonexpansive $F:Q(X)\to Q(X)$ is an isometric embedding on
$e(X)$, then it is an isometric embedding on $Q(X)$.
\end{lemma}
\begin{proof}
The quasi-pseudometric $\delta(p,q)=D(Fp,Fq)$ is at most $D$ and agrees
with $D$ on the canonical copy. Tightness gives $\delta=D$.
\end{proof}

\begin{theorem}\label{thm:inheritance}
Suppose every bijective nonexpansive self-map of $Q(X)$ preserves $e(X)$
setwise. If $X$ is $q$-plastic, then $Q(X)$ is $q$-plastic. If $X$ is
$q$-rigid, then $Q(X)$ is $q$-rigid.
\end{theorem}
\begin{proof}
The restriction of a hull map induces a bijective nonexpansive self-map of
$X$. For $q$-plastic $X$, apply Lemma~\ref{lem:propagation}; for $q$-rigid
$X$, apply Lemma~\ref{lem:relative}.
\end{proof}
These are consequences of tight-extension theory. The substantive additional
condition is invariance of the canonical copy, which the next examples show
cannot be omitted.

\subsection{Failure of inheritance on the upper real line}
Write $u(s,t)=(t-s)^+$. Its symmetrization is $|s-t|$.
\begin{proposition}\label{prop:upperline}
The space $(\R,u)$ is $q$-hyperconvex but not $q$-plastic.
\end{proposition}
\begin{proof}
For compatible double balls, the required inequalities are
$x_i-s_i\le w\le x_i+r_i$. Compatibility gives
$x_j-s_j\le x_i+r_i$ for all $i,j$, hence
\[
\sup_j(x_j-s_j)\le\inf_i(x_i+r_i).
\]
For a nonempty family, any fixed index $i_0$ bounds both quantities between
$x_{i_0}-s_{i_0}$ and $x_{i_0}+r_{i_0}$, so a real $w$ can be chosen
between them. The empty family has intersection $\R$.
Thus the space is $q$-hyperconvex. The bijection $T(t)=t/2$ satisfies
$u(Ts,Tt)=u(s,t)/2$, so it is nonexpansive and not an isometry.
\end{proof}

\begin{proposition}\label{prop:linehull}
If $A\subseteq\R$ is unbounded above and below, then $Q(A,u|_A)$ is
isometric to $(\R,u)$, with the canonical copy corresponding to $A$.
\end{proposition}
\begin{proof}
For $t\in\R$ put $f_1^t(x)=(x-t)^+$ and $f_2^t(x)=(t-x)^+$ on $A$.
Ampleness follows from $(y-x)^+\le(t-x)^++(y-t)^+$.
Choosing points of $A$ below both $t,x$, and above both $t,x$, respectively,
verifies the two extremality identities \eqref{eq:extremal}, hence minimality.

Conversely, for a minimal pair $f$, set $t=\sup_{y\in A}(y-f_1(y))$.
Ampleness gives $y-f_1(y)\le x_0+f_2(x_0)$ for any fixed $x_0\in A$;
an individual term gives a finite lower bound. Thus $t\in\R$, and
\eqref{eq:extremal} gives
\[
f_2(x)=\sup_{y\in A}((y-x)^+-f_1(y))^+=(t-x)^+.
\]
Here the inner positive part equals $(y-x-f_1(y))^+$ because
$f_1(y)\ge0$. Substitution into the other identity gives
\[
f_1(x)=\sup_{y\in A}((x-y)^+-(t-y)^+)^+=(x-t)^+.
\]
The upper bound follows from the upper quasi-metric triangle inequality;
equality is attained using $y$ below both $x,t$. Thus every minimal pair
is $f^t$. The two suprema in \eqref{eq:hullmetric} both equal $(t-s)^+$
for $D(f^s,f^t)$, using points above both parameters for the first and
below both for the second. Finally, $e(a)=f^a$ for $a\in A$.
\end{proof}

\begin{proposition}\label{prop:integers}
The upper quasi-metric integers $(\Z,u)$ are bicomplete and $q$-plastic,
but their $q$-hyperconvex hull is not $q$-plastic.
\end{proposition}
\begin{proof}
The symmetrization is the complete metric $|m-n|$. A bijective
$u$-nonexpansive $F$ is strictly increasing: for $m<n$,
$u(Fn,Fm)=0$ gives $Fm\le Fn$, and injectivity makes the inequality
strict. Surjectivity rules out gaps, so $F(n+1)=F(n)+1$.
Hence $F(n)=n+k$ for some integer $k$, and $F$ is an isometry.
Propositions~\ref{prop:linehull} and \ref{prop:upperline} identify the
hull as the non-$q$-plastic upper real line.
\end{proof}

\begin{corollary}\label{cor:rigidfailure}
There is a $q$-rigid $T_0$-quasi-metric space whose $q$-hyperconvex hull is
not $q$-plastic.
\end{corollary}
\begin{proof}
By \cite{Banakh}, there is a Euclidean-dense $A\subseteq\R$ whose
bijective Euclidean-nonexpansive self-maps are all the identity.
Every $u$-nonexpansive map is Euclidean-nonexpansive, so $(A,u)$ is
$q$-rigid. Density makes $A$ unbounded in both directions; apply
Proposition~\ref{prop:linehull}.
\end{proof}
The contraction $t\mapsto t/2$ sends the canonical integers to
$\frac12\Z$, exhibiting precisely the obstruction excluded by
Theorem~\ref{thm:inheritance}.

\subsection{An intrinsic order criterion}
For $p,q\in Q(X)$ define $p\preceq q$ by $D(q,p)=0$.
The triangle inequality and the $T_0$ condition make this a partial order.
If it is a bounded lattice, write $\bot,\top$ for its least and greatest
elements. An element $p$ is \emph{complemented} if some $q$ satisfies
$p\wedge q=\bot$ and $p\vee q=\top$; let $C(Q(X))$ denote these elements.

\begin{proposition}\label{prop:complements}
Suppose $(Q(X),D^s)$ is compact, its zero-distance order is a bounded
lattice, and $e(X)=C(Q(X))$. If $X$ is $q$-rigid, then $Q(X)$ is $q$-rigid.
\end{proposition}
\begin{proof}
By Proposition~\ref{prop:totalbounded}, a bijective nonexpansive hull map
is a $D$-isometry. It is therefore an order automorphism and preserves
$\bot,\top$, meets, joins and complemented elements. It preserves $e(X)$
setwise, so Theorem~\ref{thm:inheritance} gives $q$-rigidity.
\end{proof}
Compactness alone already gives $q$-plasticity. The added conclusion is
pointwise rigidity; the equality with complemented elements is a structural
hypothesis that must be verified, not a general property of canonical copies.

\begin{example}[A four-vertex asymmetric rectangle]\label{ex:rectangle}
Let $a,b>0$, $a\ne b$, and set
\[
B=[0,a]\times[0,b],\qquad X=\{(0,0),(a,0),(0,b),(a,b)\}.
\]
Equip both spaces with $u_2(p,q)=\max\{(q_1-p_1)^+,(q_2-p_2)^+\}$.
We identify $Q(X)$ with $B$ and verify every hypothesis of
Proposition~\ref{prop:complements}. The rectangular hull geometry is known:
\cite[Example~2]{EndpointsII} identifies $B$ as the hull of its opposite
vertices $(a,0),(0,b)$; see also \cite{Willerton}. All four vertices are
used here so that the canonical copy equals the complemented elements.

\emph{Hull identification.} Write $L_1=a,L_2=b$. For a nonempty
compatible family $(z^\lambda,r_\lambda,s_\lambda)$ in $B$, the
double-ball conditions in coordinate $i$ are
$z_i^\lambda-s_\lambda\le w_i\le z_i^\lambda+r_\lambda$ and
$0\le w_i\le L_i$. Compatibility gives
\[
\max\{0,\sup_\mu(z_i^\mu-s_\mu)\}
\le \min\{L_i,\inf_\lambda(z_i^\lambda+r_\lambda)\}.
\]
The bounds are finite, since centres are in $[0,L_i]$ and a fixed index
bounds the suprema and infima. Choose each $w_i$ between them.
The empty family has intersection $B$, so $B$ is $q$-hyperconvex.

To check tightness over $X$, let $\delta\le u_2$ be a quasi-pseudometric
agreeing with $u_2$ on $X\times X$, and set
$c_1=(0,b),d_1=(a,0),c_2=(a,0),d_2=(0,b)$.
For $p,q\in B$,
$u_2(c_i,d_i)=L_i$, $u_2(c_i,p)=p_i$ and
$u_2(q,d_i)=L_i-q_i$. Therefore
\[
L_i=\delta(c_i,d_i)
\le\delta(c_i,p)+\delta(p,q)+\delta(q,d_i)
\le p_i+\delta(p,q)+L_i-q_i.
\]
Thus $\delta(p,q)\ge(q_i-p_i)^+$ in both coordinates, giving
$\delta=u_2$. The tight-envelope characterization \cite{Tight}
identifies $B$ with $Q(X)$ by an isometry carrying $e(X)$ to $X$.

\emph{Compactness.} The symmetrization of $u_2$ is the maximum norm metric.
The closed rectangle is compact in that metric.

\emph{Bounded lattice order.} The relation
$p\preceq q\iff u_2(q,p)=0$ is coordinatewise order. Its least and
greatest elements are $(0,0)$ and $(a,b)$; meets and joins are coordinatewise
minima and maxima. It is therefore a bounded lattice.

\emph{Complemented elements.} If $p,q$ are complements, then
$\min\{p_i,q_i\}=0$ and $\max\{p_i,q_i\}=L_i$.
Hence $p_i\in\{0,L_i\}$, so $p$ is a vertex. Conversely, the complement
of a vertex $p$ is $(a-p_1,b-p_2)$. Consequently $C(B)=X$, which gives
$C(Q(X))=e(X)$ under the hull identification.

\emph{Rigidity of the original space.} A bijective nonexpansive map of
the finite set $X$ has some power $F^N=\id$. The cyclic chain
\[
u_2(x,y)\ge u_2(Fx,Fy)\ge\cdots\ge u_2(F^Nx,F^Ny)=u_2(x,y)
\]
makes it an isometry and order automorphism. It fixes bottom and top.
The two remaining vertices are the atoms, whose distances from $(0,0)$
are $a$ and $b$. These are distinct, so the atoms are fixed as well.
Thus $X$ is $q$-rigid, and Proposition~\ref{prop:complements} makes $B$
$q$-rigid. For instance, take $a=1,b=2$.

If $a=b$, the three structural hypotheses still hold, but coordinate
transposition is a nonidentity isometry of both $X$ and $B$.
This verifies the separate role of $q$-rigidity of the original space.
\end{example}

The upper integers and the rectangle illustrate different mechanisms:
bicompleteness and plasticity do not force preservation by the hull, whereas
an intrinsic lattice description can make the canonical copy invariant.
Compactness first supplies isometries, and recognition of the copy lets
rigidity of the original space control the hull.

\subsection{Further elementary illustrations}
The Sierpi\'nski pair $S=\{0,1\}$, with $v(0,1)=0,v(1,0)=1$,
has hull $([0,1],v)$ for $v(s,t)=(s-t)^+$.
Indeed, its minimal pairs are
$f_1^t(0)=t,f_1^t(1)=0,f_2^t(0)=0,f_2^t(1)=1-t$,
$0\le t\le1$, and $D(f^s,f^t)=(s-t)^+$.
Both spaces are $q$-rigid: the swap fails nonexpansiveness on the pair;
on the interval, compactness gives a Euclidean isometry, and reflection
fails directed nonexpansiveness.

More generally, put $u_n(x,y)=\max_{1\le i\le n}(y_i-x_i)^+$ on
$\R^n$. Its symmetrization is $\|x-y\|_\infty$.
The product double-ball argument used in Example~\ref{ex:rectangle}
shows $q$-hyperconvexity (without the coordinate truncations).
Homotheties give nonplasticity, every compact subspace is $q$-plastic by
Proposition~\ref{prop:totalbounded}, and Proposition~\ref{prop:dense}
gives a dense $q$-rigid subspace: each nonempty open set contains an
uncountable compact cube. These are illustrations of the preceding
mechanisms, not additional independent contributions.

\section{Directed intervals and endpoint rigidity}\label{sec:intervals}
The chain condition below is formulated in ordered distances. It gives an
endpoint-coordinate representation and total boundedness, hence automatic
$q$-plasticity. Endpoint preservation also controls nonbijective maps into
other spaces. Collinearity and endpoint geometry in quasi-metric hulls
have been studied in \cite{EndpointsII}.

\begin{definition}\label{def:interval}
A $T_0$-quasi-metric space $(I,d)$ is a \emph{two-sided directed interval}
with distinct endpoints $a,b$ if, for every ordered pair $(x,y)$, at least
one equality holds:
\begin{align}
d(a,b)&=d(a,x)+d(x,y)+d(y,b),\label{eq:chainab}\\
d(b,a)&=d(b,x)+d(x,y)+d(y,a).\label{eq:chainba}
\end{align}
\end{definition}
The orientation may depend on $(x,y)$. Under symmetry this says that
one of the chains $a,x,y,b$ and $a,y,x,b$ has length $d(a,b)$.

\begin{proposition}[Endpoint-coordinate representation]\label{prop:coordinates}
Every two-sided directed interval satisfies
\[
d(x,y)=\max\{(d(a,y)-d(a,x))^+,(d(b,y)-d(b,x))^+\}.
\]
Thus $\Phi(x)=(d(a,x),d(b,x))$ is an isometric embedding into
$(\R^2,u_2)$ with bounded image, and
$d^s(x,y)=\|\Phi(x)-\Phi(y)\|_\infty$.
\end{proposition}
\begin{proof}
The triangle inequality bounds each positive difference by $d(x,y)$.
If \eqref{eq:chainab} holds, then
\[
d(x,y)=d(a,b)-d(a,x)-d(y,b)\le d(a,y)-d(a,x),
\]
using $d(a,b)\le d(a,y)+d(y,b)$. If \eqref{eq:chainba} holds,
the same argument gives $d(x,y)\le d(b,y)-d(b,x)$.
This proves the formula. The $T_0$ condition makes $\Phi$ injective.
Every directed distance is at most $M=\max\{d(a,b),d(b,a)\}$ by
the chain condition, so $\Phi(I)\subseteq[0,M]^2$.
Taking the maximum for both orientations gives the metric identity.
\end{proof}

\begin{corollary}\label{cor:intervalplastic}
Every two-sided directed interval is $q$-plastic.
\end{corollary}
\begin{proof}
Its symmetrization is isometric to a bounded subset of the maximum norm
plane, hence is totally bounded. Apply Proposition~\ref{prop:totalbounded}.
\end{proof}

\begin{theorem}\label{thm:endpoints}
Let $(I,d)$ be a two-sided directed interval with endpoints $a,b$.
If $F:I\to(Y,\rho)$ is nonexpansive and preserves both directed endpoint
distances, then $F$ is an isometric embedding.
\end{theorem}
\begin{proof}
For an ordered pair satisfying \eqref{eq:chainab},
\begin{align*}
d(a,b)&=\rho(Fa,Fb)\\
&\le\rho(Fa,Fx)+\rho(Fx,Fy)+\rho(Fy,Fb)\\
&\le d(a,x)+d(x,y)+d(y,b)=d(a,b).
\end{align*}
Equality throughout forces $\rho(Fx,Fy)=d(x,y)$.
For \eqref{eq:chainba}, use the chain starting at $b$ and ending at $a$
and preservation of $d(b,a)$. Thus every directed distance is preserved.
\end{proof}

\begin{theorem}\label{thm:diameter}
If $(a,b)$ is the unique ordered diametral pair of a two-sided directed
interval $(I,d)$, then $I$ is $q$-rigid.
\end{theorem}
\begin{proof}
The interval is bounded by Proposition~\ref{prop:coordinates}.
Let $\Delta=\diam_d(I)=d(a,b)$ and let $F$ be bijective and nonexpansive.
Choose $x,y$ with $Fx=a,Fy=b$. Then
$\Delta=d(Fx,Fy)\le d(x,y)\le\Delta$.
Uniqueness gives $x=a,y=b$, so $F$ fixes the endpoints.
Theorem~\ref{thm:endpoints} makes $F$ an isometry, which preserves
$\Phi$. Injectivity of $\Phi$ in Proposition~\ref{prop:coordinates}
gives $F=\id$.
\end{proof}
In fact, any nonexpansive self-map fixing both endpoints is the identity:
apply Theorem~\ref{thm:endpoints} and injectivity of $\Phi$.
The unique-diameter hypothesis is one intrinsic way to force this endpoint
recognition; it is not needed for $q$-plasticity.

\begin{example}\label{ex:upperinterval}
For $a<b$, the upper interval $([a,b],u)$ is $q$-rigid.
If $x\le y$, it satisfies the chain from $a$ to $b$; if $x>y$,
the chain from $b$ to $a$ has all distances zero. Also $(a,b)$ is
the unique pair at directed distance $b-a$, so
Theorem~\ref{thm:diameter} applies.
\end{example}

\section{Limitations and further questions}\label{sec:questions}
The maximum symmetrization need not retain the relevant plasticity behaviour.
\begin{proposition}\label{prop:maxcounter}
Plasticity of $(X,d^s)$ does not imply $q$-plasticity of $(X,d)$.
\end{proposition}
\begin{proof}
On $\Z$ put $a_m=1+(1+2^m)^{-1}$ and
\[
d(m,n)=\begin{cases}0,&m=n,\\a_m,&m<n,\\2,&m>n.\end{cases}
\]
All nonzero distances lie in $(1,2]$. A triangle through a third distinct
point has right-hand side greater than $2$; repeated-point triangles are
immediate. Thus $d$ is a $T_0$-quasi-metric. Its symmetrization has constant
distance $2$ between distinct points, so every bijection is a metric isometry.
But $F(m)=m+1$ is bijective and $d$-nonexpansive: forward distances
strictly decrease and reverse distances stay $2$. It is not a $d$-isometry.
\end{proof}

The hull examples suggest the following questions.
\begin{problem}
Find intrinsic conditions under which $e(X)$ is invariant under every
bijective nonexpansive self-map of $Q(X)$. In particular, characterize
canonical minimal pairs intrinsically and find recognition criteria for
noncompact hulls. Which spaces have $q$-plastic or $q$-rigid hulls?
\end{problem}
\begin{problem}
Determine natural bicomplete asymmetric normed spaces that are q-LCDH.
For separable c-crowded quasi-metric spaces, which cardinalities and
descriptive complexities can be prescribed for dense $q$-rigid subspaces?
When can they be Borel, analytic or $\sigma$-compact?
\end{problem}
\begin{problem}
Identify geometric conditions under which $q$-plasticity of $(X,d)$ is
equivalent to plasticity of $(X,d^s)$, or under which preservation of
$d^s$ by a bijective $d$-nonexpansive map forces preservation of $d$.
\end{problem}
For the sum symmetrization, preservation already suffices by equality of
the two nonexpansive summands.

Two qualifications are necessary for questions about asymmetric normed
spaces. A nontrivial real linear subspace is never $q$-plastic, since
$x\mapsto x/2$ is a bijective contraction for $d_p(x,y)=p(y-x)$.
A nontrivial additive subgroup is never pointwise $q$-rigid, because its
translations are nonidentity isometries.
\begin{problem}
When does a separable asymmetric normed space contain a dense additive
subgroup $H$ on which every bijective $d_p$-nonexpansive map is an
isometry? When must every such map fixing $0$ be the identity, equivalently,
every such map be a translation?
\end{problem}
If $p$ is symmetric, inversion fixes $0$ and is a nonidentity isometry of
each nontrivial additive subgroup. Origin-fixing rigidity therefore requires
asymmetric geometry excluding inversion, or a weaker notion allowing it.

\begin{problem}
Characterize two-sided directed intervals whose endpoints are preserved
by every bijective nonexpansive self-map without assuming a unique ordered
diametral pair. For hulls, find conditions making a bijective nonexpansive
map uniquely determined by its restriction to the canonical copy.
\end{problem}

\section{Conclusion}
The upper-line examples show that neither $q$-plasticity together with
bicompleteness nor pointwise $q$-rigidity forces $q$-plasticity of the
$q$-hyperconvex hull. Tightness propagates distances from the canonical copy,
but invariance of that copy is an additional issue. The complemented-element
criterion resolves it under explicit compactness and lattice hypotheses,
verified by the four-vertex rectangle.

The remaining results identify the supporting mechanisms: total boundedness
gives $q$-plasticity through sum symmetrization, dense rigidity transfers from
the metric theorem, and the dense-subspace contraction argument remains
conditional on q-LCDH. Directed chain equalities give a bounded
endpoint-coordinate representation and automatic plasticity; unique ordered
diametral endpoints supply pointwise rigidity. Broader noncompact recognition
criteria and concrete asymmetric homogeneity classes remain open directions.

\par\needspace{5\baselineskip}
\section*{Declarations}
\noindent\textbf{Conflict of interest.} The author declares that there is
no conflict of interest.\\
\textbf{Funding.} The author was supported in part by the National Research
Foundation of South Africa.

\end{document}